\documentclass[12 pt]{amsart}

\usepackage{amsthm}
\usepackage{amssymb}
\usepackage{latexsym}
\usepackage{url}
\usepackage{amsmath}
\usepackage{verbatim}
\usepackage{graphicx}
\usepackage{epsf}
\usepackage{enumerate}
\usepackage[width=5.7cm]{caption}
\usepackage{geometry}
\usepackage{hyperref}
\hypersetup{breaklinks=true,
pagecolor=white,
colorlinks=true,%false
citecolor=black,%
filecolor=black,%
linkcolor=black,%
urlcolor=black
}

\newtheorem{thm}{Theorem}[section]

\newtheorem{lem}[thm]{Lemma}
\newtheorem{prop}[thm]{Proposition}

\newtheorem{defin}[thm]{Definition}

\theoremstyle{definition}
\newtheorem{expl}[thm]{Example}

\newcommand{\dom}{\textnormal{dom}}
\newcommand{\Dom}{\textnormal{Dom}}

\newcommand{\R}{\mathbb{R}}

\newcommand{\wt}{\widetilde}

  \renewenvironment{thebibliography}[1]{%
    \begin{oldthebibliography}{#1}%
      \setlength{\parskip}{0ex}%
      \setlength{\itemsep}{0ex}%
  }%
  {%
    \end{oldthebibliography}%
  }

\title[Zone diagrams]{Zone diagrams in compact subsets of uniformly convex
normed spaces}
\author{Eva Kopeck\'a}
\address{Institute of Mathematics, Czech Academy of Sciences \\
\v{Z}itn\'a 25, CZ-11567 Prague, Czech Republic \\ and
Institut f\"ur Analysis, Johannes Kepler Universit\"at\\
Altenbergerstrasse 69, A-4040 Linz, Austria.}
\email{eva@bayou.uni-linz.ac.at}
\author{Daniel Reem}
\address{The Technion - Israel Institute of Technology, 32000 Haifa, Israel. 
 Current address (July 2011): Department of Mathematics, University of Haifa, Mount Carmel, 31905 Haifa, Israel.}
\email{dream@tx.technion.ac.il}
\author{Simeon Reich}
\address{The Technion - Israel Institute of Technology, 32000 Haifa, Israel}
\email{sreich@tx.technion.ac.il}
\begin{document}
\maketitle
\begin{abstract}
A zone diagram is a relatively new concept which has emerged in computational
geometry and is related to Voronoi diagrams. Formally, it is
a fixed point of a certain mapping, and neither its uniqueness nor its existence
are obvious in advance. It has been studied by several authors,
starting with T. Asano, J. Matou{\v{s}}ek and T. Tokuyama, who considered
the Euclidean plane with singleton sites, and proved the existence and
uniqueness of zone diagrams there. In the present paper we prove the
existence of zone diagrams with respect to finitely many pairwise disjoint
compact sites contained in a compact and convex subset of a uniformly
convex normed space, provided that either the sites or the convex subset
satisfy a certain mild condition. The proof is based on the Schauder fixed
point theorem, the Curtis-Schori theorem regarding the Hilbert cube, and
on recent results concerning the characterization of Voronoi cells as a collection
of line segments and their geometric stability with respect to small
changes of the corresponding sites. Along the way we obtain the continuity
of the Dom mapping as well as interesting and apparently new properties
of Voronoi cells.
\end{abstract}
\section{Introduction}
A zone diagram is a relatively new concept related to geometry and fixed point
theory. In order to understand it better, consider first the
more familiar
concept of a Voronoi diagram. In a Voronoi diagram we start with a set $X$,
a distance function $d$, and a collection of subsets $(P_k)_{k\in K}$ in $X$  (called the sites or the  generators), and with each site $P_k$ we associate the  Voronoi cell $R_k$, that is, 
the set of all $x\in X$ the distance of which to $P_k$ is not
greater than
its distance to the union of the other sites $P_j$, $j\neq k$. On the other hand,
in a zone diagram we associate with each site $P_k$ the set $R_k$ of all
$x\in X$ the distance of which to $P_k$ is not greater than its distance
to the union of the other sets $R_j$, $j\neq k$. Figures \ref{fig:IntroVoronoi}
and  \ref{fig:IntroZone} show the Voronoi and zone diagrams, respectively,
corresponding to the same ten singleton sites in the Euclidean plane.

%%%%%%%%%%%%%%%%%%%%%%%%%%%%%%%%%%%%%%%%%%%%%%%%%%%%%%%%%%%%%%%%%%%%
\begin{figure}
\begin{minipage}[t]{0.45\textwidth}
\begin{center}
{\includegraphics[scale=0.6]{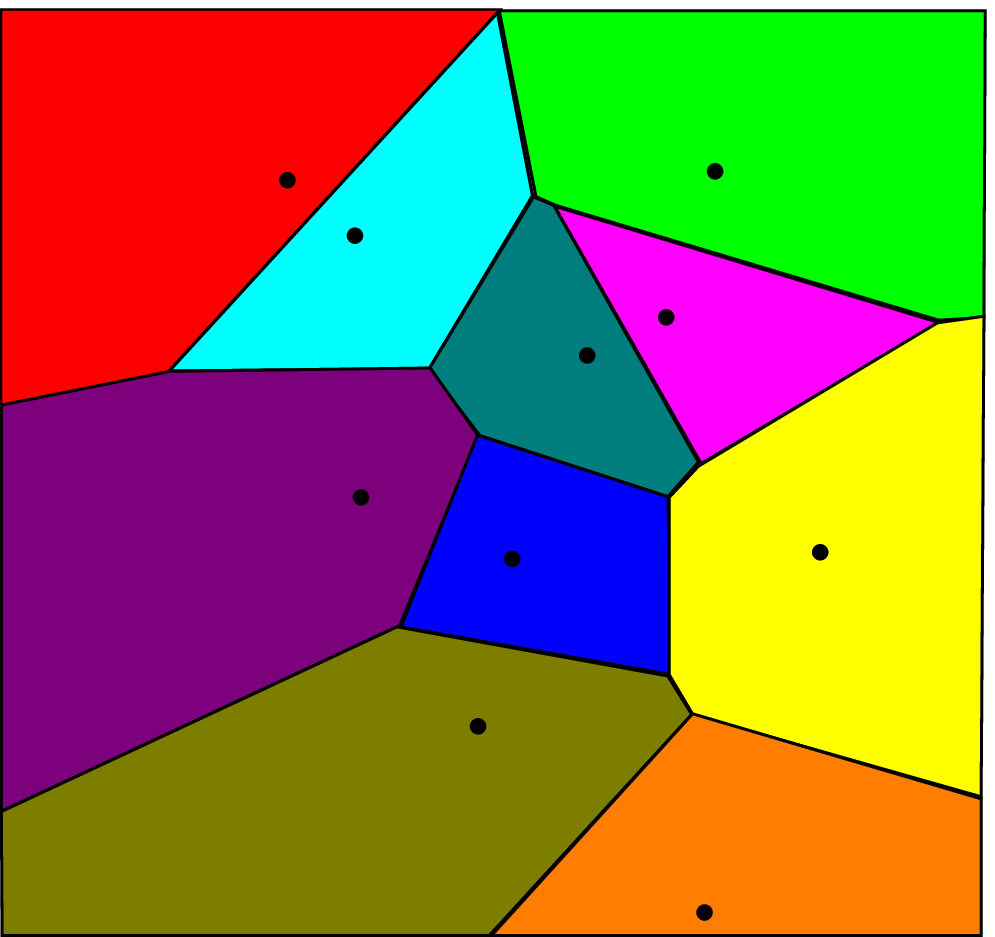}}%{alg40.png}}
\end{center}
 \caption{A Voronoi diagram of 10 point sites in a square in $(\R^2,\ell_2)$.}
\label{fig:IntroVoronoi}
\end{minipage}
\hfill
\begin{minipage}[t]{0.48\textwidth}
\begin{center}
{\includegraphics[scale=0.6]{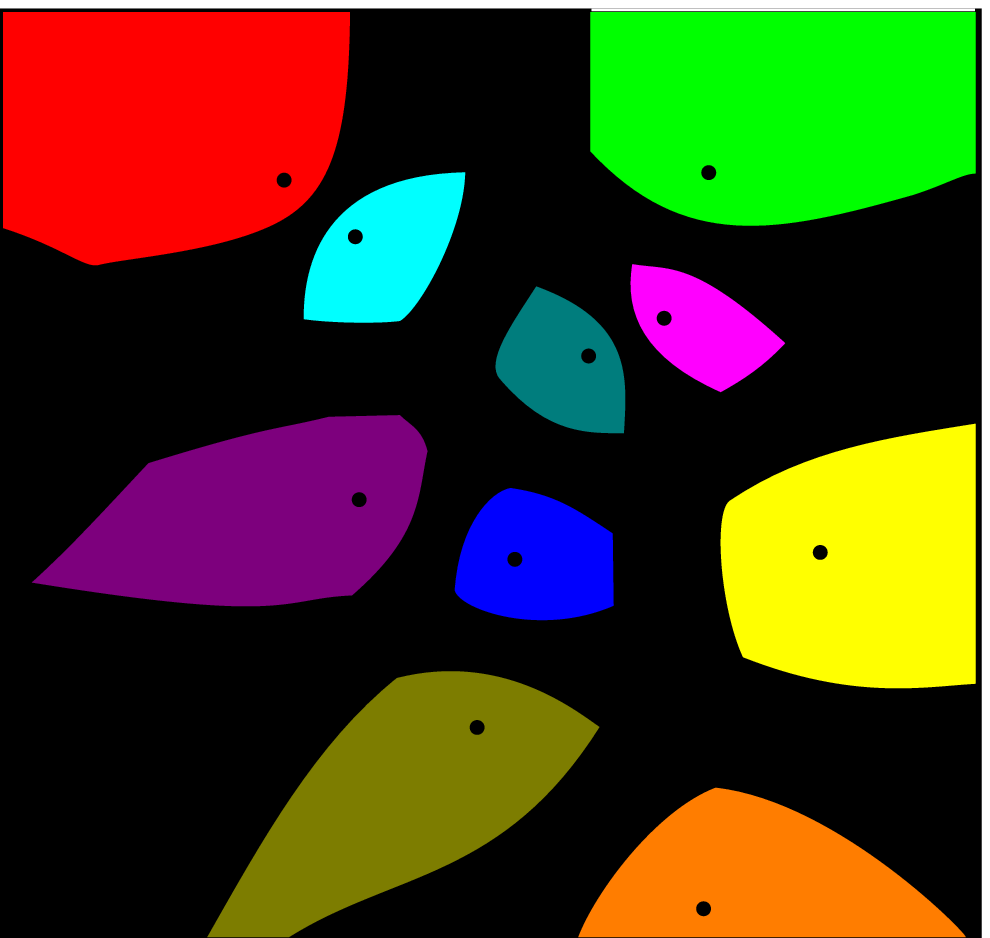}}
\end{center}
 \caption{A zone diagram of the same 10 points as in
Figure \ref{fig:IntroVoronoi}.}
\label{fig:IntroZone}
\end{minipage}
\end{figure}
%%%%%%%%%%%%%%%%%%%%%%%%%%%%%%%%%%%%%%%%%%%%%%%%%%%%%%%%%%%%%%%%%%%%
At first sight, it seems that the definition of a zone diagram is
circular, because the definition of each $R_k$ depends on $R_k$ itself via
the definition of the other cells  $R_j$, $j\neq k$. On second thought, we
see that, in fact, a zone diagram is defined to be a fixed point of a
certain mapping (called the $\Dom$ mapping), that is, a solution of a
certain equation. While the Voronoi diagram is  explicitly defined,  so its
existence (and uniqueness) are obvious, neither the existence nor the
uniqueness of a zone diagram are obvious in advance.
As a result, in addition to the problem of finding algorithms for computing
zone diagrams, we are faced with the more fundamental problem of
establishing their existence (and uniqueness) in various settings, and with the problem of reaching a better  understanding of this concept.

The concept of a zone diagram was first defined and studied
by T. Asano, J. Matou{\v{s}}ek and T. Tokuyama \cite{AMT2,AMTn} (see also  \cite{AsanoTokuyama}), in the case 
where $(X,d)$ was the Euclidean plane, each site $P_k$ was a
single point, and all these (finitely many) points were different. They proved
the existence and uniqueness of a zone diagram in this case, and also
suggested a natural iterative algorithm for approximating it.
Their proofs rely heavily  on the above setting. Several other papers  related to zone diagrams in the plane have been published, e.g., those of T. Asano and D. Kirkpatrick  \cite{AsanKirk}, of  J. Chun, Y. Okada and T. Tokuyama \cite{COT}, and 
recently of S. C. de Biasi, B. Kalantari and I. Kalantari \cite{DeBiasiKalantaris}.

Shortly after \cite{AMTn}, the authors of  \cite{ReemReichZone} considered general sites in abstract spaces, called $m$-spaces, in which $X$ is an arbitrary nonempty set and the ``distance'' function should only satisfy the condition $d(x,x)\leq d(x,y)\,\,\,\,
\forall x,y\in X$ and can take any value in the interval
$[-\infty,\infty]$. They introduced the concept of a double zone diagram,
and using it and the  Knaster-Tarski fixed point theorem, proved the
existence of a zone diagram with respect to any two sites in $X$. They also showed that in general the zone diagram is not unique.  
In a recent work by K. Imai, A. Kawamura,  J. Matou\v{s}ek, Y. Muramatsu and
T. Tokuyama \cite{IKMMT}, the existence and uniqueness of the zone diagram
with respect to any number of general positively separated sites in the
$n$-dimensional
Euclidean space $\R^n$ was announced. The proof is based on results from \cite{ReemReichZone} and on an elegant geometric argument specific to Euclidean spaces. Very recently some of these authors have generalized this result to finite dimensional normed spaces which are both strictly convex and smooth  \cite{KMT}.

In the present paper we prove the existence
of zone diagrams with respect to finitely many pairwise disjoint compact sites contained in a compact and convex 
subset of a (possibly infinite  dimensional) uniformly convex space, provided 
that either the sites or the convex subset satisfy a certain mild condition. 
This mild condition holds, for instance, if either the convex subset $X$ has 
a strictly convex boundary or the sites are contained in the interior of $X$  relative to the affine hull spanned by $X$ (and $X$ is arbitrary). 
The proof is based on the Schauder fixed point theorem, the Curtis-Schori
theorem regarding the Hilbert cube, and on recent results concerning the
characterization of Voronoi cells as a collection of line segments and their
geometric stability with respect to small changes of the corresponding sites
(see Sections \ref{sec:details} and \ref{sec:TechTools}).
Along the way we obtain the continuity of the Dom mapping
(Proposition  \ref{prop:DomContinuous}) in a general setting and
 interesting properties of Voronoi cells, namely 
Lemma \ref{lem:homeomorphism} and Lemma \ref{lem:LocConnected}. 
Although Voronoi diagrams have been the subject of extensive research during 
the last decades \cite{Aurenhammer,OBSC,VoronoiWeb}, 
this research has been mainly focused on Euclidean finite 
dimensional spaces (in many cases just $\R^2$  or $\R^3$), and it seems 
that these lemmata are new even for $\R^2$ with a non-Euclidean norm.

It may be of interest to compare our main existence result with the recent 
existence result described in \cite{KMT}. On the one hand, our result is weaker than that result, since we only prove the existence of a zone diagram in a compact and convex set, while in \cite{KMT} uniqueness is also proved and the setting is the whole space $\R^{n}$. In addition, it seems that some of the arguments in \cite{KMT}, although only formulated for finitely many compact sites, can be extended to infinitely many, positively separated closed sites. On the other hand, our result is stronger in the sense that we allow infinite dimensional spaces and we do not require the smoothness of the norm. As a matter of fact, the counterexamples mentioned in \cite{KMT} show that uniqueness does not necessarily hold if the norm is not smooth. In any case, the strategies used for proving these two results are completely different: in \cite{KMT} the authors use the existence of double zone diagrams (based on the Knaster-Tarski fixed point theorem) and several geometric arguments, and here we use the Schauder fixed point theorem, the Curtis-Schori theorem regarding the Hilbert cube, and several general results about Voronoi cells in uniformly convex normed spaces and elsewhere.

The structure of the paper is as follows: in Section \ref{sec:definitions}
we present the basic definitions and notation. In Section \ref{sec:details}
we provide the outline of the proof of the main result. 
In Section \ref{sec:TechTools} we formulate several claims which are needed 
in the proof. The proof itself is given in
Sections \ref{sec:HilbertCube}, \ref{sec:DomCont}
and \ref{sec:Main}. We conclude the paper with
Section \ref{sec:end}, which contains open problems and two pictures of 
zone diagrams.

We note that although the setting in the
main result is a compact and convex subset of a uniformly convex normed space,
many of the auxiliary results actually hold in a more general setting with 
essentially the same proofs, and therefore we formulate and prove them there.

\section{Definitions and Notation}\label{sec:definitions}
In this section we present our notation and basic definitions.
We consider a closed and convex set $X\neq \emptyset$ in some normed space
$(\widetilde{X},|\cdot|)$, real or complex, finite or infinite dimensional.
The induced metric is $d(x,y)=|x-y|$. We assume that $X$  is not a singleton, for otherwise everything is trivial. 

The notation $|\theta|=1$ will always mean a unit vector $\theta$. Given $p\in X$ we set $\Theta_p=\{\theta: |\theta|=1, p+t\theta\in X\,\,\textnormal{for some}\,\,t>0\}$. This is the set of all directions such that rays emanating from 
$p$ in these directions intersect $X$ not only at $p$. We denote by $[p,x]$ the closed line segment connecting $p$ and $x$, i.e., the set $\{p+t(x-p): t\in [0,1]\}$. We denote by $B(x,r)$ the open ball of center $x$ and radius $r$. The notation $Y\approx Z$ means that the topological spaces $Y$ and $Z$ are homeomorphic.
\begin{defin}\label{def:dom}
 Given two nonempty sets $P,A\subseteq X$, the dominance region
$\dom(P,A)$ of $P$ with respect to $A$ is the set of all $x\in X$
the distance of which to $P$ is not greater than its distance to $A$,
that is, \begin{equation*}
\dom(P,A)=\{x\in X: d(x,P)\leq d(x,A)\}.
\end{equation*}
Here $d(x,A)=\inf\{d(x,a): a\in A\}$.
\end{defin}
\begin{defin}\label{def:Voronoi}
Let $K$ be a set of at least 2 elements (indices), possibly
infinite. Given a  tuple $(P_k)_{k\in K}$ of nonempty subsets
$P_k\subseteq X$, called the generators or the sites, the Voronoi diagram  induced by this tuple is the tuple $(R_k)_{k\in K}$ of nonempty subsets
$R_k\subseteq X$, such that for all $k\in K$,
\begin{equation*}
R_k=\dom(P_k,{\underset{j\neq k}\bigcup P_j})=\{x\in X: d(x,P_k)\leq d(x,P_j)\,\,\forall j\neq k,\,j\in K \}.
\end{equation*}
 In other words, each $R_k$, called a Voronoi cell, is the set of
all $x\in X$ the distance of which to $P_k$ is not greater than
its distance to the union of the other $P_j$, $j\neq k$. 
\end{defin}
\begin{defin}
Let $(X,d)$ be a metric space and let $K$ be a set of at least 2 elements (indices), possibly
infinite. Given a
 tuple $(P_k)_{k\in K}$ of nonempty subsets
$P_k\subseteq X$, a zone diagram with respect to
that tuple is a tuple $R=(R_k)_{k\in K}$ of nonempty subsets
$R_k\subseteq X$ such that
\begin{equation*}
R_k=\dom(P_k,\textstyle{\underset{j\neq k}\bigcup R_j})\quad \forall k\in K.
\end{equation*}
In other words, if we define $X_k=\{C: P_k\subseteq C\subseteq X\}$, then a zone diagram
is a fixed point of the mapping $\Dom:\underset{{k\in K}}\prod
X_k\to \underset{{k\in K}}\prod
X_k$, defined by % $\Dom: Y\to Y$
\begin{equation}\label{eq:TZoneDef}
\Dom(R)=(\dom(P_k,\textstyle{\underset{j\neq k}\bigcup R_j}))_{k\in K}.
\end{equation}
\end{defin}
For example, let $P_1=\{p_1\}$ and $P_2=\{p_2\}$ be two different sets in the Euclidean plane, or more generally, in a Hilbert space $X$. In this case the corresponding Voronoi diagram consists of two half-spaces:  $\dom(P_1,P_2)$ is the half-space containing $P_1$ and determined by the
hyperplane passing through the middle of the line segment $[p_1,p_2]$ and perpendicular to it, and $\dom(P_2,P_1)$ is the other half-space. The zone  diagram of $(P_1,P_2)$ exists by the results of \cite{AMTn} (in the case of  the Euclidean plane), or \cite{ReemReichZone} (in the general case),
but it  is not clear how to describe the zone diagram explicitly.

We now recall the definition of strictly and uniformly convex spaces.
\begin{defin}\label{def:UniformlyConvex}\label{page:UniConvDef}
A normed space $(\widetilde{X},|\cdot|)$ is said to be strictly convex
if for all $x,y\in \wt{X}$ satisfying $|x|=|y|=1$ and $x\neq y$,
the inequality $|(x+y)/2|<1$ holds. $(\widetilde{X},|\cdot|)$ is said to be
uniformly convex if for any $\epsilon\in (0,2]$, there exists $\delta\in
(0,1]$ such that for all $x,y\in \wt{X}$, if $|x|=|y|=1$ and $|x-y|\geq \epsilon$, then $|(x+y)/2|\leq 1-\delta$.
\end{defin}
Roughly speaking, if the space is uniformly convex, then for any
$\epsilon>0$, there exists a uniform positive lower bound on how deep the
midpoint between any two unit vectors must penetrate the unit ball,
assuming the distance  between them is at least $\epsilon$.
In general normed spaces the penetration is not necessarily positive,
since the unit ball may contain line segments. The plane $\R^2$ endowed with the 
max norm
$|\cdot|_{\infty}$ is a typical example of this.  A uniformly convex space
is always strictly convex, and if it is also  finite dimensional,
then the converse is true too. The $n$-dimensional Euclidean  space $\R^n$,
or more generally, inner product spaces, the sequence spaces  $\ell_p$,
the Lebesgue spaces $L_p(\Omega)$, $p\in (1,\infty)$, and a uniformly convex
product of a finite number of uniformly convex spaces,
are all examples of uniformly convex spaces. See \cite{Clarkson} and, for
instance, \cite{LindenTzafriri} and \cite{GoebelReich}
for more information regarding uniformly convex spaces.

We now recall three definitions of a topological character. 
\begin{defin}\label{def:HilbertCube}
The Hilbert cube $I^{\infty}$ is the set
$I^{\infty}=\prod_{n=1}^{\infty}[0,1/n]$ as a topological space
the topology of which is induced by the $\ell_2$ norm, or, equivalently, by the product topology.
\end{defin}
\begin{defin}\label{def:LocConnect}
A topological space $X$ is said to be locally (path) connected if for any $x\in X$ and any open set $U$ containing $x$, there exists a (path) connected
open set $V\subseteq U$ such that $x\in V$. $X$ is said to be weakly locally (path) connected, or (path) connected im kleinen, if for any $x\in X$ and any open set $U$ containing $x$, there exists a (path) connected
 set $C\subseteq U$ and an open set $V\subseteq C$ such that $x\in V$.
\end{defin}
\begin{defin}\label{def:Hausdorff}
Let $(X,d)$ be a metric space. Given two nonempty sets $A_1,A_2\subseteq X$, the Hausdorff distance between them is defined by
\begin{equation*}
D(A_1,A_2)=\max\{\sup_{a_1\in A_1}d(a_1,A_2),\sup_{a_2\in A_2}d(a_2,A_1)\}.
\end{equation*}
\end{defin}
Recall that the Hausdorff distance is different from the usual distance between two sets which is defined by 
\begin{equation*}
d(A_1,A_2)=\inf\{d(a_1,a_2): a_1\in A_1, a_2\in A_2\}. 
\end{equation*}
Recall also that if $(X,d)$ is compact and we consider the set of all its nonempty  closed subsets, then this space is a compact metric space with the Hausdorff  distance as its metric \cite{IllanesNadler}.

We end this section with a definition which is pertinent to the formulation of
our main result and some of our auxiliary assertions. It is followed by a brief
discussion.
\begin{defin}\label{def:emanation}
Let $X$ be a closed and  convex subset of a normed space. Let $p\in X$.  Let $\theta\in\Theta_p$. Let $L(\theta)\in (0,\infty]$ be the length of the 
line segment generated from the intersection of $X$ and the ray emanating from $p$  in the direction of $\theta$. The point $p$ is  said to have the emanation property (or to satisfy the  emanation condition)  in the direction of $\theta$ if for each $\epsilon>0$ there exists $\beta>0$ such that for any $\phi\in\Theta_p$, if $|\phi-\theta|<\beta$, then the intersection of $X$ and the ray emanating from $p$  in the direction of $\phi$ is a line segment of length at least $L(\theta)-\epsilon$. 
In other words, $L(\phi)\geq L(\theta)-\epsilon$. The point $p$ is said to have  the emanation property if it has the emanation property in the direction of every $\theta\in\Theta_p$. A subset $C$ of $X$ is said to have  the emanation property if each $p\in C$ has the emanation property. 
\end{defin}

The following examples illustrate the emanation property. In the first four the emanation property holds, and in the last one it does not hold.
\begin{expl}
$X$ is any bounded closed convex set and $p\in X$ is an arbitrary point in the interior of $X$ relative to the affine hull spanned by $X$.
\end{expl}
\begin{expl}
The boundary of the bounded closed and  convex $X$ is strictly convex (if $a\neq b$ are two points in the boundary, then the open line segment $(a,b)$ is contained in the interior of $X$ relative to the affine hull spanned by $X$) and $p\in X$ is arbitrary. Any ball in a strictly convex space has a strictly convex boundary.
\end{expl}
\begin{expl}
$X$ is a cube (of any finite dimension) and $p\in X$ is arbitrary.
\end{expl}
\begin{expl}
$X$ is a closed linear subspace and $p\in X$  is arbitrary. 
\end{expl}
\begin{expl}\label{ex:NonEmanation}
This example shows that the emanation condition does not hold in general. Consider the Hilbert space $\ell_2$. Let $(e_n)_{n=1}^{\infty}$ be the standard basis. Let $y_1=e_1$ and for each  $n>1$ let $y_n=e_1/2+e_n/n$. Let $A=\{-e_1\}$.  Let $X$ be the closed convex hull generated by $A\bigcup\{y_n: n=1,2,\ldots\}$. Let $p=0$ and let  $\theta_n=y_n/|y_n|$ for each $n$. Then $p$ does not satisfy the emanation  condition in the direction of $\theta_1$. Indeed, $\lim_{n\to\infty}\theta_n=\theta_1$ but $L(\theta_n)=\sqrt{0.25+1/n^2}<0.99=L(\theta_1)-0.01$ for each $n>1$. The  subset $X$ is in fact compact since the sequence $(y_n)_{n=1}^{\infty}$  converges. 
\end{expl}
For more details about the emanation property, see  \cite{ReemGeometricStabilityArxiv}. 

\section{Outline of the proof of the main result}\label{sec:details}
In this section we outline the proof of our main result,
Theorem \ref{thm:CompactSites}. It states that there exists a zone diagram
with respect to finitely many pairwise disjoint compact sites in a compact and convex subset of a uniformly convex normed space, provided that either the sites or the convex subset satisfy a certain mild condition. The idea of 
the proof is to find a certain space $Y$ homeomorphic to
the Hilbert cube $I^{\infty}$  ($Y\approx I^{\infty}$ for short) such that
$\Dom(Y)\subseteq Y$ and  $\Dom$ is continuous on $Y$.
Now, if $h:I^{\infty} \to Y$ is a homeomorphism, then $f=h^{-1}\circ
\Dom \circ h: I^{\infty}\to I^{\infty}$ is a continuous mapping which
maps a compact and convex subset of $\ell_2$ into itself, so the
Schauder fixed point theorem \cite{Schauder} (see also \cite[p. 119]{GranasDugondji} and Theorem \ref{thm:Schauder} below) ensures
that $f$ has a fixed point $q\in I^{\infty}$. By taking $R=h(q)$, we see
that $R$ is a fixed point of $\Dom$, that is, $R$ is a zone diagram. In
order to apply this idea, one has to find the set $Y$, to prove that it is
homeomorphic to $I^{\infty}$, and to prove the continuity of $\Dom$ on $Y$.
It has turned out that even in the case of singleton sites in a square in the
Euclidean plane the proof is not obvious (the main difficulty is to prove
the  continuity of $\Dom$), and, in fact, such a proof has never been
published.

The above strategy  was suggested by the first author, and was briefly
mentioned in \cite[p. 1188]{AMTn}. The space $Y$ was taken
to be $\prod_{k\in K} Y_k$,
where $K$ was finite, $Y_k$ was $\{C: P_k\subseteq C\subseteq Q_k\,\,
\textnormal{and}\,\,C\,\,\textnormal{is closed}\}$ and  $Q_k$ was the intersection of the Voronoi cell of $P_k$ with $X$ (the square). Since each site $P_k$ is taken to be a singleton, it follows that each $Q_k$ is actually convex,
so, in particular, it is  a connected and locally connected compact metric
space. Since, in addition, $P_k\neq Q_k$, it follows from the theorem of
D. Curtis and  R. Schori \cite[Theorem 5.2] {CurtisSchori}  stated below
(see also \cite[p. 91]{IllanesNadler})
that $Y_k$, as a metric space endowed with  the Hausdorff metric,
is homeomorphic to $I^{\infty}$. The topology on $Y$ is the product topology,
induced by the uniform Hausdorff metric
$\wt{D}((S_k)_{k\in K},(S'_k)_{k\in K})=\max\{D(S_k,S'_k): k\in K\}$,
so $Y$, as a finite  product of spaces homeomorphic to $I^{\infty}$, is also
homeomorphic to $I^{\infty}$.

\begin{thm}\label{thm:Schauder}{\bf (Schauder)}
Let $X$ be a nonempty convex and compact subset of a normed space. If $f:X
\to X$ is continuous, then it has a fixed point.
\end{thm}
\begin{thm}\label{thm:CurtisSchori}{\bf (Curtis-Schori)}
Let $X$ be a Peano continuum, that is, a connected and locally connected compact metric space,
and let $P\subseteq X$, $P\neq X$ be closed and nonempty.
Let $2^X_P=\{C: P\subseteq C\subseteq X, \, C\,\,\textnormal{is closed}\}$, endowed with the Hausdorff metric. Then $2^X_P\approx I^{\infty}$.
\end{thm}
In the general case, the application of the above strategy, and, in
particular, the verification of the hypotheses of Theorem  \ref{thm:CurtisSchori}, are not a simple task, and
they require several additional tools related to dominance regions, 
such as their characterization as unions of line segments,
and their stability with respect to small perturbations of the
relevant sets.
These results will be stated in the next section. 
They have recently been established in \cite{ReemISVD09,ReemGeometricStabilityArxiv},
and their proofs can be found there. Using these results, we first prove the existence of a zone diagram with respect to finite sites, and then, approximating  compact sets by finite subsets of them and applying a continuity argument, we extend this  existence result to any compact sites with the emanation property.

\section{Several technical tools}\label{sec:TechTools}
In this section we either prove or cite several technical claims needed for establishing the main result; see \cite{ReemGeometricStabilityArxiv} and  \cite{ReemISVD09} for those proofs not included here. 

The following theorem is a new representation theorem for dominance regions.
\begin{thm}\label{thm:domInterval}
Let $X$ be a closed and convex subset of a normed space. Let $P,A\subseteq X$ be nonempty. Suppose that the distance between $x$ and $P$ is attained for all $x \in X$. Then $\dom(P,A)$ is a union of
line segments starting at the points of $P$. More precisely,
given $p\in P$ and $|\theta|=1$, let
\begin{equation}\label{eq:Tdef}
T(\theta,p)=\sup\{t\in [0,\infty): p+t\theta\in X\,\,\mathrm{and}\,\,
 d(p+t\theta,p)\leq d(p+t\theta,A)\}.
\end{equation}
Then
\begin{equation*}\label{eq:dom}
\dom(P,A)=\bigcup_{p\in P}\bigcup_{|\theta|=1}[p,p+T(\theta,p)\theta].
\end{equation*}
When $T(\theta,p)=\infty$, the notation $[p,p+T(\theta,p)\theta]$ means the ray $\{p+t\theta: t\in [0,\infty)\}$.
\end{thm}
The proof of Theorem \ref{thm:domInterval} is based on the following
simple observation, which will also be needed for a different purpose
later (see the proof of Lemma \ref{lem:LocConnected}).
\begin{lem}\label{lem:segment}
Let $(\widetilde{X},|\cdot|)$ be a normed space, and let $\emptyset\neq A\subseteq \widetilde{X}$. Suppose that $y,p\in \widetilde{X}$ satisfy $d(y,p)\leq d(y,A)$. Then
$d(x,p)\leq d(x,A)$ for any $x\in [p,y]$.
\end{lem}
The next two theorems describe a continuity property of dominance regions 
and the mapping $T$ defined in \eqref{eq:Tdef} in 
uniformly convex normed spaces. We note that condition \eqref{eq:BallRho} 
below expresses the fact that the set $A$ is ``well distributed in $X$''. 
It obviously holds when $X$ is bounded, but it may also hold even when $X$ 
is not bounded, as in the case where $X=\R^n$ and $A$ is the lattice of points with 
integer coordinates.
\begin{thm}\label{thm:StabilityUC}
Let $X$ be a closed and convex subset of a uniformly convex normed space. Then, under certain conditions, the mapping $\dom(\cdot,\cdot)$ has a uniform continuity property with respect to the Hausdorff distance. More precisely, assume that $P$ and  $A$ are nonempty with $d(P,A)>0$. Suppose that 
\begin{equation}\label{eq:BallRho}
\exists \rho\in (0,\infty)\,\, \textnormal{such that}\,\,\forall x\in X\,\,\textnormal{the open ball}\,\, B(x,\rho)\,\,\textnormal{intersects} \,\,A. 
\end{equation}
Then for each  $\epsilon\in (0,d(P,A)/6)$ there exists $\Delta>0$ such that if $D(A,A')<\Delta$, $D(P,P')<\Delta$, and the distances  between any point $x\in X$ and both $P$ and $P'$ are attained,  then the inequality $D(\dom(P,A),\dom(P',A'))<\epsilon$ holds.
\end{thm}
\begin{thm}\label{thm:T}
Let $X$ be a closed and convex subset of a uniformly convex
normed space. Let $p\in X$ and suppose that $p$ has the emanation property. 
Then the mapping $T(\cdot,p)$ has a certain continuity property. More precisely, let  $A\subset X$ be nonempty such that \eqref{eq:BallRho} holds and that $d(p,A)>0$.  Then for each $\epsilon\in (0,d(p,A)/6)$ and each $\theta \in \Theta_p$ there exists $\Delta>0$ such that for each  $\phi \in \Theta_p$, if $|\theta-\phi|<\Delta$, then $|T(\theta,p)-T(\phi,p)|\leq \epsilon$. In addition, the range of the mapping $T$ is bounded by $\rho$. 
\end{thm}
We remark in passing that in the proofs of Theorems \ref{thm:StabilityUC} and \ref{thm:T}, the uniform convexity of the space enters, {\em inter alia}, through Clarkson's strong triangle inequality \cite[Theorem 3]{Clarkson}.

The following lemma will be needed for proving the local connectedness of 
certain dominance regions (Lemma \ref{lem:homeomorphism} and Lemma 
\ref{lem:LocConnected}). Note that the set $B_p$ is not necessarily the unit  ball (or a ball of some closed affine hull), as in the case where $X$ is the Hilbert cube $I^{\infty}$ in $\ell_2$, or more generally, a compact and convex subset of an infinite dimensional normed space.
\begin{lem}\label{lem:B}
Let $X$ be a convex set in a normed space and let $p\in X$. Define the set  $B_p=\{r\theta: r\in [0,1],\,\theta\in \Theta_p\}$. Then $B_p$ is convex.
\end{lem}
\begin{proof}
Let $b_1,b_2\in B_p$. Then $b_i=r_i\theta_i$ for some $r_i\in [0,1]$ and 
$|\theta_i|=1$. Let $b=\lambda_1 b_1+\lambda_2 b_2$ for 
some $\lambda_1,\lambda_2\in [0,1]$ with $\lambda_1+\lambda_2=1$. 
Suppose that $b\neq 0$, for otherwise $b\in B_p$. 
Then $b=r\theta$ with $r=|\lambda_1 b_1+\lambda_2 b_2|$ 
and $\theta=(\lambda_1 b_1+\lambda_2 b_2)/r$. 
Clearly, $0\leq r\leq r_1\lambda_1+r_2\lambda_2\leq 1$, 
and $|\theta|=1$, so it remains to prove that $p+t\theta\in X$ for some $t>0$.

By the definition of $\theta_i$, there is $t_i>0$ such that 
$y_i:=p+t_i\theta_i\in X$. We can assume that $\lambda_i\neq 0$ 
and $r_i\neq 0$ for $i\in \{1,2\}$, for otherwise $\theta=\theta_j$ 
for $j\neq i$ and $p+t_j\theta\in X$. 
Let $\beta_1=\lambda_1 r_1 t_2/(\lambda_1r_1t_2+\lambda_2r_2t_1)$, 
$\beta_2=\lambda_2r_2t_1/(\lambda_2 r_2 t_1+\lambda_1 r_1 t_2)$ 
and $t=(\beta_1 t_1+\beta_2 t_2)r/(\lambda_1 r_1+\lambda_2 r_2)$. 
These values were obtained by equating the coefficients of the $\theta_i$ 
in the equation $p+t\theta=\beta_1 y_1+\beta_2 y_2$. 
It is easy to check that 
$\beta_i\in (0,1)$, $\beta_1+\beta_2=1$, $t>0$ 
and $p+t\theta=\beta_1 y_1+\beta_2 y_2$. 
Thus $p+t\theta\in X$ because $X$ is convex.
\end{proof}
The next two lemmata are probably known, at least in a version close in  its spirit to their formulation (see, e.g., \cite[p. 162, exercise 6]{Munkres}  and \cite[Proposition~ 10.7, pp. 82-83]{IllanesNadler}), but we include their 
proof for the sake of completeness.
\begin{lem}\label{lem:WeakStrongConnect}
If $Z$ is a topological space which is weakly locally (path) connected, then it is locally (path) connected.
\end{lem}
\begin{proof}
The proof is based on the fact that $Z$ is locally (path) connected if and
only if for every open set $U$ of $Z$ each (path) connected component of
$U$  is open in $Z$ \cite[p. 161]{Munkres}. Given an open set $U$, let $x\in U$. Let $C_x$ be the (path) connected component of $x$ in $U$. Since
$Z$ is weakly locally (path) connected, there exist a (path) connected set
$C\subseteq U$ and an open set $V\subseteq C$ such that $x\in V$.
By the maximality of $C_x$, we have $C\subseteq C_x$,
and hence $V\subseteq C_x$. Thus $x$ belongs to the interior of $C_x$, and
in the
same way all the other points of $C_x$ are in its interior. Hence $C_x$ is open in $X$ and the same is true for all other (path) components of $U$.
\end{proof}
\begin{lem}\label{lem:FiniteLocConnect}
Let $(Z,d)$ be a metric space and suppose that $Z=\bigcup_{i=1}^m Z_i$ for some locally (path) connected closed subsets $Z_i$ of $Z$. Then $Z$ is locally (path) connected.
\end{lem}
\begin{proof}
It suffices to prove the assertion for $m=2$; the general case follows
by induction. Let $x\in Z$. Then $x\in Z_i$ for some $i$. Let $U$ be an
open neighborhood of $Z$ with $x\in U$. Assume first that $x\notin Z_j$ for $j\neq i$.
Then the intersection $U\cap B(x,d(x,Z_j)/2)\subseteq Z_i$  contains
an open (path) connected subset $V$ of $Z_i$ with $x\in V$,
since $Z_i$ is locally (path) connected.
Note that $V$ is also (path) connected with respect to the topology of $Z$.
Since $d(V,Z_j)>0$, it follows that $V$ is actually open in $Z$, so $x$
has an open (path) connected neighborhood contained in $U$ in this case.

Assume now that $x\in Z_1\cap Z_2$. Then $U\cap Z_i$ is
an open neighborhood of $x$ in $Z_i$, so for each $i\in \{1,2\}$, there is an
open (path) connected subset $V_i\subseteq U\cap Z_i$ of $Z_i$
with $x\in V_i$, because $Z_i$ is locally (path) connected.
The union $V_1\cup V_2$ is a (path) connected subset of $Z$ which is contained
in $U$, but it is not clear whether it is open in $Z$. However,
by definition, $V_i=V'_i\cap Z_i$ for some open set $V'_i$ of $Z$, and
\begin{equation*}
x\in V'_1\cap V'_2=(V'_1\cap V'_2\cap Z_1)\cup (V'_1\cap V'_2\cap Z_2)
\subseteq V_1\cup V_2\subseteq U.
\end{equation*}
Since $x$ is arbitrary, this proves that $Z$ is weakly locally (path)
connected, so by Lemma \ref{lem:WeakStrongConnect} it is locally (path) connected.
\end{proof}
The following lemma will be useful for proving the main result of Section \ref{sec:HilbertCube}.
\begin{lem}\label{lem:DistanceToP}
Let $X$ be a convex set in a normed space, and suppose that $P,A\subset X$ satisfy $r:=d(P,A)>0$. Let $B(P,r/2):=\{x\in X: d(x,P)<r/2\}$. Then $B(P,r/2)\subseteq \{x\in X: d(x,P)<d(x,A)\}$ and  $d(x,P)\geq r/2$ for any $x\in X$ satisfying $x\in \dom(A,P)$.
\end{lem}
\begin{proof}
The set $B(P,r/2)$ is contained in
$\{x\in X: d(x,P)<d(x,A)\}$ because if $x\in B(P,r/2)$ and $a\in A$, then $r\leq d(a,P)\leq d(a,x)+d(x,P)<d(a,x)+r/2$, so $d(x,P)<r/2\leq d(x,A)$.

Now let $x\in \dom(A,P)$. Since $d(x,A)\leq d(x,P)$, this point does not
belong to $B(P,r/2)$
by the above paragraph. Let $p\in P$ be arbitrary, and consider the line segment $[x,p]$. It intersects the boundary of $B(P,r/2)$ (otherwise the connected space $[x,p]$ would have a decomposition as a union of two disjoint open sets) at some point $y$, and since $B(P,r/2)$ is open, it follows that $d(y,P)\geq r/2$. Hence $d(x,p)=d(x,y)+d(y,p)\geq r/2$, i.e., $d(x,P)\geq r/2$, as claimed.
\end{proof}

We finish with the following lemma, the proof of which is a simple
consequence of the definition.
\begin{lem}\label{lem:GeneralDistance}
The equality $\dom(\bigcup_{i=1}^n P_i,A)=\bigcup_{i=1}^n\dom(P_i,A)$ holds for any nonempty subsets $P_1,\ldots,P_n$ and $A$ of the metric space $X$.
\end{lem}

\section{The space homeomorphic to the Hilbert cube}\label{sec:HilbertCube}
The main result of this section is Proposition \ref{prop:HilbertCube} below. It is based on Lemma ~\ref{lem:LocConnected} which shows that $\dom(p,A)$ is locally path connected whenever $d(p,A)>0$, and on Lemma  \ref{lem:FiniteConnected} which generalizes this result to $\dom(P,A)$ for any finite set $P$. As Lemma \ref{lem:segment} shows, $\dom(p,A)$ is star-shaped. However, this property by itself is not sufficient for concluding that it is locally (path) connected, since there are simple examples  of star-shaped sets in $\R^2$ which are not locally connected. In fact, a result of T. Zamfirescu \cite[Theorem 4]{Zamfirescu} shows that a large  class (in the sense of Baire category) of star-shaped sets in $\R^n$ are not locally connected. 

As a result, the local path connectedness of $\dom(p,A)$ must be proved. It has turned out that the proof, given in Lemma \ref{lem:LocConnected}, is somewhat technical, and in a special but important case, namely Lemma \ref{lem:homeomorphism} below, a simpler proof can be presented. This special case is of interest in itself, and its proof also casts some light on the strategy for proving Lemma \ref{lem:LocConnected}. The condition on $p$ and $s$ described in  Lemma \ref{lem:homeomorphism} holds, for instance, when $p$ is in the interior of $X$ relative to the closed affine hull spanned by it. See also the short discussion before Lemma \ref{lem:B} regarding the set $B_p$. 

\begin{lem}\label{lem:homeomorphism}
Let $X$ be a closed and convex subset of a uniformly convex normed space. Let $p\in X$. Suppose that there exists some $s>0$ such that $p+s\theta\in X$ for each $\theta\in \Theta_p$. Suppose also that $p$ has the emanation property. Let $\emptyset \neq A\subset X$ be such that $d(p,A)>0$. Suppose also that condition \eqref{eq:BallRho} holds. Then $\dom(p,A)$ is homeomorphic to the set $B_p=\{r\theta: r\in [0,1],\,\theta\in \Theta_p\}$, and, in particular, $\dom(p,A)$ is path connected and locally path connected.
\end{lem}
\begin{proof}
Let $f:B_p \to \dom(p,A)$ be defined by
$f(r\theta)=p+rT(\theta,p)\theta$,
i.e., $f(0)=p$ and $f(x)=p+T(x/|x|,p)x$ for $x\neq 0$, where $T(\theta,p)$
is defined in \eqref{eq:Tdef}. Let
$\epsilon=\min\{d(p,A)/2,s/2\}$. Let $\theta\in \Theta_p$ and
$t\in [0,\epsilon]$ be arbitrary. Then for $x=p+t\theta$ we have $x\in X$
because
$t\leq s$ and $X$ is convex. In addition, $2\epsilon \leq d(p,A)\leq
d(p,x)+d(x,A)\leq \epsilon+d(x,A)$, so $d(x,p)\leq \epsilon\leq d(x,A)$
and hence $x\in \dom(p,A)$. Thus $T(\theta,p)\geq \epsilon$
by the definition of $T$. From Theorem \ref{thm:T} it
follows that $T$ is bounded, so $f$ is well defined. By
Theorem \ref{thm:domInterval}, $f$ is onto, and it is one-to-one by a
direct calculation.  If $y=p+rT(\theta,p)\theta$, then $\theta=(y-p)/|y-p|$ and $r=|y-p|/T(\theta,p)$, so the inverse function is defined by
\begin{equation*}
f^{-1}(y)=\left\{\begin{array}{ll}
\frac{y-p}{T((y-p)/|y-p|,p)} & y\neq p\\
0 & y=p.
\end{array}\right.
\end{equation*}
It now follows from Theorem \ref{thm:T} that both $f$ and
$f^{-1}$ are continuous. Since $B_p$ is convex by Lemma \ref{lem:B}, it is path connected and locally path connected, so $\dom(p,A)$ is both path
connected and locally path connected, as claimed.
\end{proof}
%%%%%%%%%%%%%%%%%%%%%%%%%%%%%%%%%%%%%%%%%%%%%%%%%%%%%%%%%%%%%%%%%%%%%%%%%%
\begin{figure}
%\hfill
\begin{minipage}[t]{0.38\textwidth}
\begin{center}
{\includegraphics[scale=0.8]{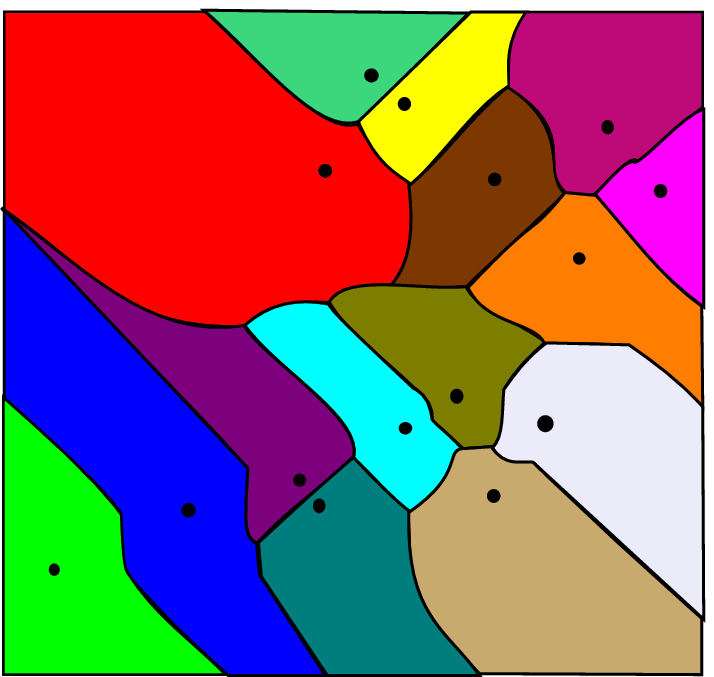}}
\end{center}
 \caption{An example related to Lemma \ref{lem:homeomorphism} in a square in $(\R^2,\ell_{7})$. Here $p=p_k$, $A$ is the collection of the other singleton sites $p_j,
 j\neq k$, and $\dom(p,A)$ is the Voronoi cell of $p_k$. The set $B_p$ is simply the unit ball.}
\label{fig:LemmaBall}
\end{minipage}%
\hfill
\begin{minipage}[t]{0.38\textwidth}
\begin{center}
{\includegraphics[scale=0.5]{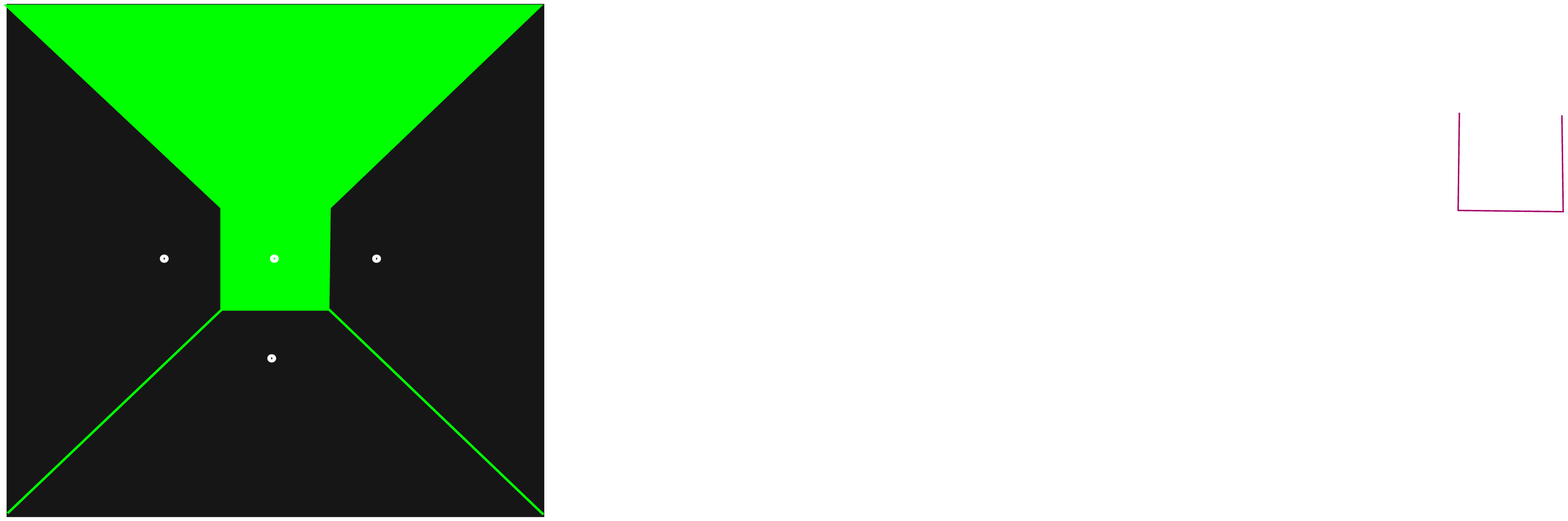}}
\end{center}
 \caption{An example where the conclusion of Lemma \ref{lem:homeomorphism} fails if the uniform convexity assumption is removed. Here $X$ is a square in $(\R^2,\ell_{\infty})$, $p=(0,0)$, and $A=\{(2,0),(-2,0),(0,-2)\}$. The green region is $\dom(p,A)$.}
\label{fig:InStability000} \label{page:InStability000}
\end{minipage}%
\end{figure}
%%%%%%%%%%%%%%%%%%%%%%%%%%%%%%%%%%%%%%%%%%%%%%%%%%%%%%%%%%%%%%%%%%%%%%%%%%%%%
An example related to Lemma \ref{lem:homeomorphism} is given in Figure \ref{fig:LemmaBall}. This may be somewhat surprising, but in general the conclusion of Lemma \ref{lem:homeomorphism} is not true, and Figure \ref{fig:InStability000} presents a
counterexample in the non-uniformly convex space $(\R^2,\ell_{\infty}$) with two simple sets. 

An examination of the above proof suggests that problems can appear if no  positive number $s$ satisfies the condition in the formulation of Lemma \ref{lem:homeomorphism}, because in this case the function $f^{-1}$ may not be continuous. Such a phenomenon occurs in infinite dimensional spaces, for instance, when $X$ is the Hilbert cube $I^{\infty}$ in $\ell_2$, but cannot happen if $p$ is in the interior of $X$ relative to the closed affine hull spanned by it. In order to replace Lemma \ref{lem:homeomorphism} one has to show directly by other arguments that $\dom(p,A)$ is connected
and locally connected. This will be done in Lemma \ref{lem:LocConnected}  below.

\begin{lem}\label{lem:LocConnected} 
Let $X$ be a closed and convex set in a uniformly convex normed space.
Let  $p\in X$ and $\emptyset \neq A\subset X$ and suppose that $d(p,A)>0$.
Suppose also that condition \eqref{eq:BallRho} holds and that $p$ has the emanation property. Then $\dom(p,A)$ is path connected and locally path  connected.
\end{lem}
\begin{proof}
By Lemma \ref{lem:segment}, any two points $x_1,x_2\in \dom(p,A)$
can be connected via $p$, so $\dom(p,A)$ is path connected. In order
to prove that
$\dom(p,A)$ is locally path connected it suffices by Lemma \ref{lem:WeakStrongConnect} to show that it is weakly locally path connected. Let $x_0\in \dom(p,A)$. If $x_0=p$, then let $V=B(x_0,r)\cap X$ for $r=d(p,A)/3$. Given $y\in V$, the triangle inequality shows that $3r=d(p,A)\leq d(p,y)+d(y,A)\leq r+d(y,A)$. Hence $d(y,p)\leq r<2r\leq d(y,A)$ and $y\in \dom(p,A)$, so $V\subseteq \dom(p,A)$. Since $V$ is convex, $x_0=p$ has a path connected open neighborhood.

Now suppose that $x_0\neq p$, and for any $\epsilon>0$ let $U=B(x_0,\epsilon)\cap \dom(p,A)$ be any neighborhood of $x_0$. Let $\theta_0=(x_0-p)/|x_0-p|$. Then $|\theta_0|=1$, and since $x_0=p+|x_0-p|\theta_0\in \dom(p,A)$, the definition of $T(\theta_0,p)$ (in \eqref{eq:Tdef}) implies that $d(x_0,p)\leq T(\theta_0,p)$. Hence $x'_0:=(x_0-p)/T((x_0-p)/|x_0-p|,p)=(|x_0-p|/T(\theta_0,p))\theta_0$
is well defined, different from 0 and belongs to $B_p=\{r\theta: r\in [0,1],\,|\theta|=1,\,\,p+t\theta\in X \,\textnormal{for some}\,\, t>0\}$. In addition, $|x'|>0$ for each $x'\in B(x'_0,|x'_0|/2)\cap B_p$.

Define $f:B(x'_0,|x'_0|/2)\cap B_p\to \dom(p,A)$ by
$f(x')=p+T(x'/|x'|,p)x'$. Since $T(\cdot,p)$ is continuous (Theorem 
\ref{thm:T}) and since $f(x'_0)=x_0$, it follows
that $f$ is continuous, so for the given $\epsilon>0$, there exists some
$\delta_1\in (0,|x'_0|/2)$ such that if $|x'-x'_0|<\delta_1$,  then
$|f(x')-f(x'_0)|<\epsilon$. Hence $f(U'_1)\subset  U$ for
$U'_1=B(x'_0,\delta_1)\cap B_p$. Since $U'_1$ is a path connected (in
fact, convex) open set containing $x'_0$, the continuity of $f$ implies
that $f(U'_1)$ is a path connected set contained in $U$ and containing
$x_0=f(x'_0)$. In order to establish the weak (path) connectedness of $\dom(p,A)$, it suffices to show that $f(U'_1)$ contains an open neighborhood of $x_0$.

Since $T(\theta_0,p)>0$, the continuity of $T(\cdot,p)$ implies that  there exists $\delta_2\in (0,1)$ such  that if $\theta\in \Theta_p$ and  $|\theta-\theta_0|<\delta_2$, then $T(\theta,p)>0$.
 Let $\delta_3=\delta_2|x_0-p|/2$. Given  $x\in B(x_0,\delta_3)\cap \dom(p,A)$, we have
\begin{multline}\label{eq:ThetaCont}
\left|\frac{x-p}{|x-p|}-\frac{x_0-p}{|x_0-p|}\right|=
\frac{\left||x_0-p|(x-p)-|x-p|(x_0-p)\right|}{|x-p||x_0-p|}=\\
\frac{\left|(|x_0-p|-|x-p|)(x-p)+|x-p|((x-p)-(x_0-p))\right|}{|x-p||x_0-p|}
\leq \frac{2|x_0-x|}{|x_0-p|}<\delta_2.
\end{multline}
Thus $T((x-p)/|x-p|,p)>0$ for each $x\in B(x_0,\delta_3)\cap \dom(p,A)$.
For each such $x$, let $g(x)=(x-p)/T((x-p)/|x-p|,p)$. Since $g$ is
continuous, for $\delta_1$ from the definition of $U'_1$ there exists
$\delta_4\in (0,\delta_3)$ such that if $|x-x_0|<\delta_4$, then
$|g(x)-g(x_0)|<\delta_1$. But $g(x_0)=x'_0$ and $g(x)\in U'_1$ for each
$x$ in the set $V:=B(x_0,\delta_4)\cap\dom(p,A)$. In addition, by the
definition of $f$ and $g$ we have $x=f(g(x))$ for each $x\in V$ (just let
$\theta=(x-p)/|x-p|$ and $x'=g(x)$, and note that $x'/|x'|=\theta$). Hence
$V$ is an open neighborhood of $x_0$ which is contained in $f(U'_1)$, so
$\dom(p,A)$ is weakly locally path connected and hence locally path connected by Lemma \ref{lem:WeakStrongConnect}.
\end{proof}

\begin{lem}\label{lem:FiniteConnected}
Let $X$ be a closed and convex subset of a uniformly convex normed space.
Let
$P=\{p_1,\ldots, p_m\}\subset X$ and assume that $\emptyset\neq A\subset X$ satisfies $d(P,A)>0$. Suppose also that 
condition \eqref{eq:BallRho} holds and that the subset $P$ has the emanation property. Then $\dom(P,A)$ is locally path connected.
\end{lem}
\begin{proof}
This assertion is a simple consequence of the facts that
$\dom(P,A)=\bigcup_{i=1}^m\dom(p_i,A))$ (by Lemma \ref{lem:GeneralDistance}),  that $\dom(p_i,A)$ is locally path connected (by Lemma
\ref{lem:LocConnected}) and closed for each $i$, and the fact that the metric space $\dom(P,A)$ is a finite union of locally path connected closed sets and hence locally path connected by Lemma \ref{lem:FiniteLocConnect}. Note that the topology of $\dom(p_i,A)$, which is induced by the norm of the space,  coincides with its topology as a subspace of $\dom(P,A)$, and hence we can apply Lemma \ref{lem:FiniteLocConnect}. 
\end{proof}

\begin{prop}\label{prop:HilbertCube}
Let $X$ be a compact and convex subset of a uniformly convex normed space,
and let  $(P_k)_{k\in K}$ be a finite tuple of finite sets in $X$
which are pairwise disjoint. Suppose that for each $k\in K$ the site $P_k$ has the  emanation property. For each $k\in K$, let $A_k=\bigcup_{j\neq
k}P_j$ and  $Q_k=\dom(P_k,A_k)$. Let $Y_k=\{C: \,C\,\,\textnormal{is closed}\,\,\textnormal{and}\,\, P_k\subseteq C\subseteq Q_k \}$, endowed with the Hausdorff metric. Let $Y=\prod_{k\in K}Y_k$, endowed with the uniform Hausdorff metric. Then $Y\approx I^{\infty}$.
\end{prop}
\begin{proof}
Given $k\in K$ and $p\in P_k$, let $C_p$ be the path connected 
component of $p$ in $Q_k$, and let $E_p=\{q\in P_k: q\in C_p\}$. Since $d(P_k,A_k)>0$ for each $k$ by assumption, we have $d(q,A_k)>0$ for any $q\in E_p$, so $\dom(q,A_k)$ is path   connected by Lemma \ref{lem:LocConnected}. Thus $\dom(q,A_k)\subseteq C_p$ for each $q\in E_p$, and hence $\dom(E_p,A_k)\subseteq C_p$ by Lemma \ref{lem:GeneralDistance}. On the other hand, if $x\in C_p$, then $x\in \dom(p',A_k)$ for some $p'\in P_k$, because $Q_k=\bigcup_{p\in P_k}\dom(p,A_k)$ by Lemma \ref{lem:GeneralDistance}. Lemma  \ref{lem:LocConnected} implies that there is a path between $p'$ and $x$, and since $x\in C_p$, there is a path between $x$ and $p$. Thus $p'\in E_p$, so $x\in \dom(p',A_k)\subseteq \dom(E_p,A_k)$ and hence $\dom(E_p,A_k)=C_p$.
 
Since for each $p_1,p_2\in P_k$, either 
$E_{p_1}=E_{p_2}$ or $E_{p_1}\cap E_{p_2}=\emptyset$, 
it follows that $\{E_p : p\in P_k\}=\{E_{k,1},\ldots,E_{k,m_k}\}$ for some 
disjoint sets $E_{k,l}, \,l=1,\ldots, m_k$, the union of which is $P_k$. Hence the path connected components of $Q_k$ are $\dom(E_{k,1},A_k),\ldots,\dom(E_{k,m_k},A_k)$, and since they are closed and disjoint in the compact set $X$, they are, in fact, compact and positively separated. Since  $E_{k,l}\subseteq P_k$, we have $d(E_{k,l},A_k)\geq d(P_k,A_k)>0$ for each $k$ and each $l$. Hence, by Lemma  \ref{lem:FiniteConnected}, the sets $\dom(E_{k,l},A_k)$ are also locally path connected.

By Lemma  \ref{lem:DistanceToP}, 
we have $E_{k,l}\varsubsetneqq\dom(E_{k,l},A_k)$ for each $l=1,\ldots,m_k$  
and each $k$, because there are points $x$ with $d(x,E_{k,l})=r/4$ 
for $r=d(E_{k,l},A_k)$ and each such point is in $\dom(E_{k,l},A_k)$, 
but not in $E_{k,l}$. Consequently, Theorem ~\ref{thm:CurtisSchori} 
implies that  $2^{\dom(E_{k,l},A_k)}_{E_{k,l}}\approx I^{\infty}$. 
Using the fact that the sets $\dom(E_{k,1},P_k)$, $\dom(E_{k,2},P_k)$,$\ldots,\dom(E_{k,m_k},P_k)$ 
are positively separated (and also the fact that any
closed set $C$ in $Y_k$ has the unique decomposition 
$C=\bigcup_{l=1}^{m_k} (C\cap \dom(E_{k,l},A_k))$ as a disjoint union of 
closed sets contained in the components $\dom(E_{k,l},A_k)$), 
we can easily verify that $Y_k$ with the Hausdorff metric is  homeomorphic 
to the finite product space $\prod_{l=1}^{m_k}2^{\dom(E_{k,l},A_k)}_{E_{k,l}}$,
endowed  with the uniform Hausdorff metric. Therefore $Y_k\approx 
(I^{\infty})^{m_k}\approx I^{\infty}$, and hence $Y\approx I^{\infty}$.
\end{proof}

\section{Continuity of the Dom mapping}\label{sec:DomCont}
\begin{prop}\label{prop:DomContinuous}
Let $X$ be a convex subset of a uniformly convex normed space, and let
$(P_k)_{k\in K}$ be a tuple of nonempty and positively separated sets in $X$, that is,  $\inf\{d(P_k,P_j): j,k\in K, j\neq k\}>0$. Assume that condition \eqref{eq:BallRho} holds (with the same $\rho$) for each $k\in K$ with $P_k$ and $A_k=\bigcup_{j\neq k}P_j$ instead of $P$ and $A$. Suppose that the distance between each $x\in X$ and each $P_k,\,k\in K$, is attained. For each $k\in K$ let $Q_k=\dom(P_k,\bigcup_{j\neq k}P_j)$ and $Y_k=\{C: \,\,C\,\textnormal{is closed}\,\,\textnormal{and}\,\, P_k\subseteq C\subseteq Q_k \}$.
 Let $Y=\prod_{k\in K}Y_k$, endowed with the uniform Hausdorff metric 
$\wt{D}$ defined by $\wt{D}((S_k)_{k\in K},(S'_k)_{k\in K})
= \sup\{D(S_k,S'_k): k\in K\}$. Then  $\Dom$ maps $Y$ into itself and
is uniformly continuous there. 
\end{prop}
\begin{proof}
 First, note that for any nonempty subsets $P,A,B$ of $X$, if $A\subseteq B$, then  $\dom(P,B)\subseteq \dom(P,A)$. Now let $S=(S_k)_{k\in K}\in Y$ be given. By the definition of $Y_k$, we have $P_k\subseteq S_k$ for each $k\in K$, so $W_k:=\dom(P_k,\bigcup_{j\neq k}S_j)\subseteq \dom(P_k,\bigcup_{j\neq k}P_j)=Q_k$ by the definition of $Q_k$. Since the $k$-th component of  $W:=\Dom(S)$ is  the closed subset $W_k$, we conclude that $(P_k)_{k\in K}\subseteq \Dom(S)\subseteq (Q_k)_{k\in K}$, and hence $\Dom$ maps $Y$ into itself.

Now let $\epsilon>0$ satisfy $12\epsilon\leq\inf\{d(P_k,P_j): j,k\in K, j\neq k\}$, and let $\Delta$ correspond to $\epsilon$ in Theorem \ref{thm:StabilityUC}. Let $S=(S_k)_{k\in K},S'=(S'_k)_{k\in K}$ be any two tuples in $Y$ satisfying $\wt{D}(S,S')<\Delta$. As a result, we also have $D(S_k,S'_k)<\Delta$ for each $k\in K$. This implies that $D(\bigcup_{j\neq k}S_j,\bigcup_{j\neq k}S'_j)\leq \Delta$ for each $k\in K$, because $D(\bigcup_{j\neq k}S_j,\bigcup_{j\neq k}S'_j)\leq \sup\{D(S_j,S'_j): j\neq k\}$.

Fix $k\in K$. If $x\in \bigcup_{j\neq k}S_j$, then $x\in S_j$ for some $j\neq k$, so $x\in Q_j=\dom(P_j,\bigcup_{i\neq j}P_i)\subseteq \dom(P_j,P_k)$. Hence $d(x,P_k)\geq 6\epsilon$ by  Lemma  \ref{lem:DistanceToP}. Thus
 $d(P_k,\bigcup_{j\neq k}S_j)\geq 6\epsilon$ for each $k\in K$.
Hence  all  the conditions of Theorem \ref{thm:StabilityUC} are satisfied  
(here $P=P_k=P'$, $A=\bigcup_{j\neq k}S_j$, $A'=\bigcup_{j\neq k}S'_j$ and 
condition \eqref{eq:BallRho} of Theorem \ref{thm:StabilityUC} for $A$ 
follows from condition \eqref{eq:BallRho} for  $A_k=\bigcup_{j\neq k}P_j$ 
because $P_j\subseteq S_j$). Thus
$D(\dom(P_k,\bigcup_{j\neq  k}S_j),\dom(P_k,\bigcup_{j\neq k}S'_k))<\epsilon$ for each $k\in K$. By the definition of $\Dom$ and $\wt{D}$, we have $\wt{D}(\Dom(S),\Dom(S'))\leq \epsilon$, so $\Dom$ is indeed uniformly continuous on $Y$.
\end{proof}

\section{The main result}\label{sec:Main}
\begin{prop}\label{prop:ZoneFiniteOrder}
Let $X$ be a compact and convex subset of a uniformly convex
normed space, and let $(P_k)_{k\in K}$ be a finite tuple of finite sets in $X$ which are pairwise disjoint. Suppose that for each $k\in K$, the site $P_k$ has the emanation property. Then there exists a zone diagram with respect to these sites.
\end{prop}
\begin{proof}
Given the finite sites $(P_k)_{k\in K}$, let $Y$ be as in Proposition  \ref{prop:DomContinuous}. Since $\Dom$ maps $Y$ into itself and is continuous there by Propostion \ref{prop:DomContinuous}, and since $Y$ is homeomorphic to the Hilbert cube by Proposition \ref{prop:HilbertCube},
we deduce the existence of a zone diagram from the Schauder fixed
point theorem, as explained in the beginning of Section \ref{sec:details}.
\end{proof}
\begin{thm}\label{thm:CompactSites}
Let $X$ be a compact and convex subset of a uniformly convex
normed space, and let $(P_k)_{k\in K}$ be a finite tuple of compact sets in $X$  which are pairwise disjoint. Suppose that for each $k\in K$, the site $P_k$ has the emanation property. Then there exists a zone diagram with respect to these sites.
\end{thm}
\begin{proof}
Let $(P_k)_{k\in K}$ be the given finite tuple of pairwise disjoint compact sites.
Let $r=\min\{d(P_k,P_j): j\neq k\}$. By the compactness of the sites, 
for each positive integer $m$ and for each $k\in K$, 
there exists a finite subset  $P_{m,k}$ of $P_k$ 
such that $D(P_{m,k},P_k)<1/m$. 
Note that $d(P_{m,k},P_{m,j})\geq r$ whenever $j\neq k$. For each $m$, let 
$R_m=(R_{m,k})_{k\in K}$ be a zone diagram corresponding to the tuple $(P_{m,k})_{k\in K}$, the existence of which is guaranteed by Proposition \ref{prop:ZoneFiniteOrder}. In other words,
\begin{equation}\label{eq:R_mk}
R_{m,k}=\dom(P_{m,k},\bigcup_{j\neq k}R_{m,j})\quad \forall k\in K.
\end{equation}
Since the space of all nonempty compact subsets of $X$ (endowed with the
Hausdorff metric) is compact \cite{IllanesNadler}, and since $K$ is finite,  we can find a convergent subsequence of the finite sites and the components of the corresponding zone diagrams. Therefore  for some subsequence $(m_l)_{l=1}^{\infty}$ of positive integers the sequence $(R_{m_l,k})_{l=1}^{\infty}$ converges  to some set $R_k$, and the sequence  $(P_{m_l,k})_{l=1}^{\infty}$ converges to $P_k$ by the definition of $P_{m,k}$. Hence $\bigcup_{j\neq k}R_{m_l,j}$ converges to $\bigcup_{j\neq k}R_{j}$, because $D(\bigcup_{j\neq k}R_{m_l,j},\bigcup_{j\neq k}R_{j})\leq \max \{D(R_{m_l,j},R_j): j\neq k\}$.

Let $\epsilon'=r/12$. Then $12\epsilon' \leq \min\{d(P_{m_l,k},P_{m_l,j}): j\neq k\}$. Let $k\in K$. If $j\neq k$ and $x\in \bigcup_{j\neq k}R_{m_l,j}$, then $x\in R_{m_l,j}$ for some $j\neq k$. Therefore 
\begin{equation*}
x\in \dom(P_{m_l,j},\bigcup_{i\neq j}R_{m_l,i})\subseteq \dom(P_{m_l,j},\bigcup_{i\neq j}P_{m_l,i})\subseteq \dom(P_{m_l,j},P_{m_l,k}). 
\end{equation*}
Hence $d(x,P_{m_l,k})\geq 6\epsilon'$ by Lemma  \ref{lem:DistanceToP}, and, as a result, $d(P_{m_l,k},\bigcup_{j\neq k}R_{m_l,j})\geq 6\epsilon'$ for each $k\in K$ and $l$. Thus $d(P_k,\bigcup_{j\neq k}R_{j})\geq 6\epsilon'$ and hence the positive distance condition in Theorem \ref{thm:StabilityUC} holds. From   \eqref{eq:R_mk} and the continuity of $\dom(\cdot,\cdot)$ (Theorem ~\ref{thm:StabilityUC}) we conclude that $R=(R_k)_{k\in K}$ is a zone diagram with respect to $(P_k)_{k\in K}$.
\end{proof}
The hypotheses of Theorem \ref{thm:CompactSites} are satisfied, in particular, when all the compact sites are contained in the interior of a closed ball (or another compact and convex set) in $\R^n$.

\section{Concluding remarks and open problems}\label{sec:end}
In this short section we describe several interesting questions and 
directions for further investigation.

 An interesting problem is whether the dominance region $\dom(P,A)$ is
locally connected for general positively separated sets $P$ and $A$. If
so, then this will suggest an alternative method for proving the existence
of a zone diagram with respect to compact sites. Second, it would be interesting to extend the existence result further, 
say to all normed spaces, without any compactness requirement on the sites 
and the  subset $X$. The question of uniqueness is interesting too. It 
seems that the uniform convexity assumption on the norm is not sufficient 
even in the plane with two singleton sites, as shown in \cite{KMT}, but 
perhaps, following \cite{KMT},  uniform convexity combined with uniform smoothness of the norm will imply uniqueness in the infinite dimensional case too. It would also be of interest to establish our main theorem (Theorem \ref{thm:CompactSites})  and auxiliary results (e.g., Lemmata \ref{lem:homeomorphism} and  \ref{lem:LocConnected}) without imposing the condition of the emanation property on the points. 

Another interesting and natural problem is how to approximate a zone
diagram in the setting described in this paper. It turns out that there
is a way to do it, and a key point in carrying out this task is to use the
algorithm for computing Voronoi diagrams of general sites in general spaces
described in  \cite{ReemISVD09,ReemPhD}. For the sake of completeness,
we include two pictures of zone diagrams in the plane with two different  norms.
%%%%%%%%%%%%%%%%%%%%%%%%%%%%%%%%%%%%%%%%%%%%%%%%%%%%%%%%%%%%%%%%%%%%%%%%%%%%%
\begin{figure}
\begin{minipage}[t]{0.45\textwidth}
\begin{center}
{\includegraphics[scale=0.73]{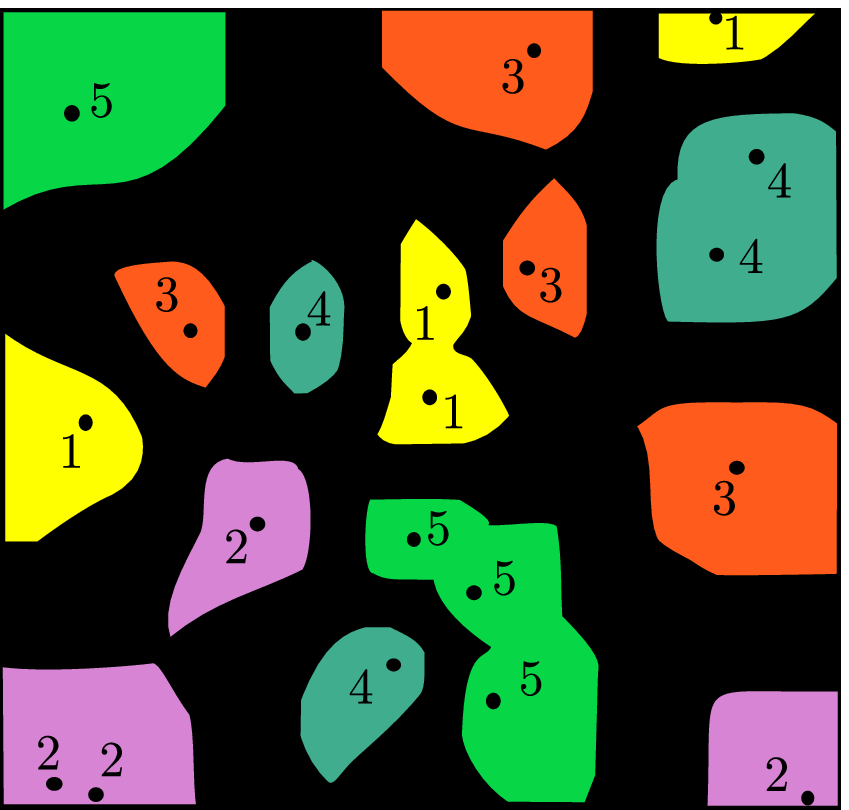}}
\end{center}
 \caption{A zone diagram of 5 sites, each with 4 points, in a square in  $(\R^2,\ell_6)$.}
\label{fig:ZD-5gen-4in-003-6}
\end{minipage}
\hfill
\begin{minipage}[t]{0.45\textwidth}
\begin{center}
{\includegraphics[scale=0.8]{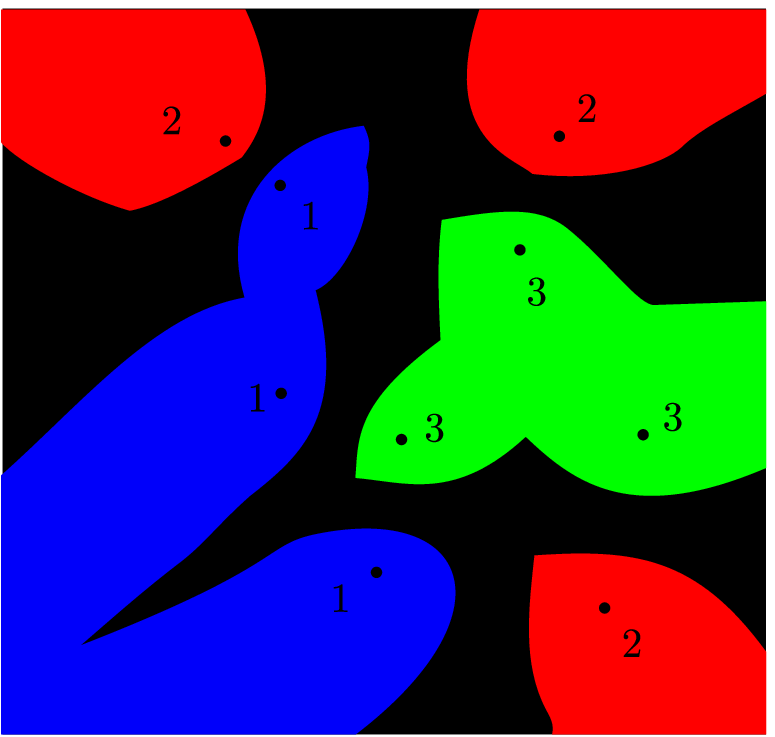}}
\end{center}
 \caption{A zone diagram  of 3 sites in a square in $(\R^2,\ell_p),\,p=3.14159$, each
with 3 points.}
\label{fig:ZD-3gen-3in-03-314159}
\end{minipage}
\end{figure}
%%%%%%%%%%%%%%%%%%%%%%%%%%%%%%%%%%%%%%%%%%%%%%%%%%%%%%%%%%%%%%%%%%%%%%%%%%%%%
The proof of Theorem ~\ref{thm:CompactSites} shows that one can approximate
the given compact sites by finite subsets of them and then the
corresponding zone diagram approximates the real one. Unfortunately, no
error estimates are obtained in the proof, and it would indeed be of 
interest to find such  estimates.
%\newpage
\vspace{0.5cm}\\
\noindent{\bf Acknowledgements}\\
\noindent We thank  Akitoshi Kawamura,  Ji\v{r}\'{i} Matou\v{s}ek, and
Takeshi Tokuyama for providing a copy of their recent paper \cite{KMT} before posting it on the arXiv. The first author was supported by Grant
FWF-P19643-N18. The third author was partially
supported by the Israel Science Foundation (Grant 647/07), the Fund for
the Promotion of Research at the Technion (Grant 2001893), and by the
Technion President's Research Fund (Grant 2007842). All the authors thank Orr Shalit and the referee for their helpful suggestions and corrections. We are very grateful to Eva Goldman for kindly redrawing the figures so expertly.

\bibliographystyle{amsplain}
\bibliography{biblio}
\vspace{0.5cm}

\end{document}